\documentclass[11pt]{amsart}

\usepackage{amssymb,geometry,stmaryrd,color}

\usepackage{amsmath}
\usepackage{amsfonts}
\usepackage{fullpage}

\usepackage{hyperref}
\hypersetup{colorlinks,linkcolor={red},citecolor={blue},urlcolor={blue}}

\usepackage{graphicx}
\usepackage{caption}

\begin{document}

\newtheorem{theorem}{Theorem}[section]
\newtheorem{lemma}[theorem]{Lemma}
\newtheorem{proposition}[theorem]{Proposition}
\newtheorem{corollary}[theorem]{Corollary}
\newtheorem{conjecture}[theorem]{Conjecture}

\theoremstyle{definition}
\newtheorem{definition}[theorem]{Definition}
\newtheorem{example}[theorem]{Example}
\newtheorem{remark}[theorem]{Remark}
\newtheorem{condition}[theorem]{Condition}
\newtheorem{assumption}[theorem]{Assumption}
\newtheorem{Notation}[theorem]{Notation}
\newtheorem{question}[theorem]{Question}
\newtheorem{Argument}[theorem]{Argument}

\numberwithin{equation}{section}

\def\s{{\bf s}} 
\def\t{{\bf t}} 
\def\u{{\bf u}} 
\def\x{{\bf x}} 
\def\y{{\bf y}} 
\def\z{{\bf z}} 
\def\B{{\bf B}} 
\def\C{{\bf C}} 
\def\K{{\bf K}}
\def\M{{\bf M}}
\def\Nn{{\bf N}}
\def\G{{\bf \Gamma}} 
\def\W{{\bf W}}
\def\X{{\bf X}}
\def\U{{\bf U}}
\def\V{{\bf V}}
\def\Un{{\bf 1}}
\def\Y{{\bf Y}}
\def\Z{{\bf Z}}
\def\P{{\bf P}}
\def\Q{{\bf Q}}
\def\L{{\bf L}}

\def\cB{{\mathcal{B}}} 
\def\cC{{\mathcal{C}}} 
\def\cD{{\mathcal{D}}} 
\def\cG{{\mathcal{G}}} 
\def\cK{{\mathcal{K}}} 
\def\cL{{\mathcal{L}}} 
\def\cR{{\mathcal{R}}} 
\def\cS{{\mathcal{S}}}
\def\cU{{\mathcal{U}}}
\def\cV{{\mathcal{V}}} 
\def\cX{{\mathcal X}}
\def\cY{{\mathcal Y}}
\def\cZ{{\mathcal Z}}

\def\Ea{E_\a}
\def\eps{{\varepsilon}} 
\def\esp{{\mathbb{E}}} 
\def\Ga{{\Gamma}}

\def\lacc{\left\{}
\def\lcr{\left[}
\def\lpa{\left(}
\def\lva{\left|}
\def\racc{\right\}}
\def\rpa{\right)}
\def\rcr{\right]}
\def\rva{\right|}

\def\prst{{\leq_{st}}}
\def\prost{{\prec_{st}}}
\def\prcvx{{\prec_{cx}}}
\def\Rr{{\bf R}}

\def\CC{{\mathbb{C}}}
\def\EE{{\mathbb{E}}}
\def\NN{{\mathbb{N}}} 
\def\QQ{{\mathbb{Q}}} 
\def\PP{{\mathbb{P}}}
\def\ZZ{{\mathbb{Z}}}
\def\RR{{\mathbb{R}}}

\def\Tt{{\bf \Theta}}
\def\Ttt{{\tilde \Tt}}

\def\A{{\bf A}}
\def\AA{{\mathcal A}}
\def\hAA{{\hat \AA}}
\def\hL{{\hat L}}
\def\hT{{\hat T}}

\def\claw{\stackrel{(d)}{\longrightarrow}}
\def\elaw{\stackrel{(d)}{=}}
\def\pslaw{\stackrel{a.s.}{\longrightarrow}}
\def\qed{\hfill$\square$}

\newcommand*\pFqskip{8mu}
\catcode`,\active
\newcommand*\pFq{\begingroup
        \catcode`\,\active
        \def ,{\mskip\pFqskip\relax}%
        \dopFq
}
\catcode`\,12
\def\dopFq#1#2#3#4#5{%
        {}_{#1}F_{#2}\biggl[\genfrac..{0pt}{}{#3}{#4};#5\biggr]%
        \endgroup
}

\def\ii{{\rm i}}
\def\S{\mathbf{S}}
\def\F{\mathbf{T}}
\def\W{\mathbf{W}}

\title[Stable directed polymer]{The intermediate disorder regime for stable directed polymer in Poisson environment}

\author[M.~Wang]{Min Wang}

\address{School of Mathematics and Statistics, Wuhan University, Wuhan 430072, China }

\email{minwang@whu.edu.cn}



\begin{abstract} We consider the stable directed polymer in Poisson random environment in dimension 1+1, under the intermediate disorder regime. We show that, under a diffusive scaling involving different parameters of the system, the normalized point-to-point partition function of the polymer converges in law to the solution of the stochastic heat equation with fractional Laplacian and Gaussian multiplicative noise.
 
\end{abstract}

\maketitle

\section{Introduction}
%\subsection{History}  \hspace{1cm}   \\
The directed polymer model was first introduced by Huse and Henley in the physical literature \cite{HuseHenley} to study interfaces of the Ising model with random impurities, and was first reformulated as the polymer measure in the mathematical literature \cite{ImbrieSpencer} by Imbrie and Spencer. This model is a description of stretched chains subject to random impurities. It has attracted much attention in recent years, because of its connections to the stochastic heat equation and the Kardar-Parisi-Zhang (KPZ) equation. Some nice surveys of polymer models are found in \cite{giacomin2007}, \cite{BergeHDR2019} and \cite{Comets2017}.

In the discrete case, the directed polymer is described by a random probability distribution on the path space of random walks on the $d$-dimensional lattice, while on each site of the lattice, a random variable is placed. The polymer is attracted to or repelled by sites, depending on the sign of the random environment. Plenty of works (\cite{AmirCorwinQuastel2011}, \cite{BergerLacoin2021}, \cite{BergerTorri2019}, \cite{CaravennaSunZygouras2017}, \cite{Comets2006} , \cite{ChenGao2023}, \cite{coscoAIHP2021}, \cite{DeyZygouras2016}, \cite{RANG2020}, \cite{RangSongWang2022}, etc.) can help to understand this discrete model. 
When the discrete random walks are replaced by Brownian motion and the environment becomes a space-time white noise, a continuum directed random polymer was constructed by Alberts, Khanin and Quastel \cite{AKQ12}. Some properties (H\"older continuity and quadratic variation properties and others) of the continuum random polymer were also studied.  

Comets and Yoshida \cite{CometsYoshida2005} firstly studied Brownian directed polymers in Poisson random environment. They proved that, if $d \leq 2$, the strong disorder occurs whenever $\beta \neq 0.$ Miura, Tawara and Tsuchida \cite{MiuraTawaraTsuchida2008} studied symmetric L\'evy directed polymers in Poisson random environment, and gave conditions for strong and weak disorder in terms of the L\'evy exponent. In particular, for a symmetric $\alpha$-stable process, if $d = 1$ and $1< \alpha < 2$, then for any $\beta \neq 0,$ the strong disorder occurs.

In 2014, Alberts, Khanin and Quastel \cite{AlbertsKhaninQuastel2014} firstly considered the intermediate disorder regime, which is a new disorder regime between the weak and strong disorder regimes. Cosco \cite{COSCO2019805} considered the intermediate disorder regime for Brownian directed polymers in Poisson environment. In the present paper, we generalize Cosco's model \cite{COSCO2019805} and consider symmetric $\alpha$-stable (with $\alpha \in (1,2]$) directed polymer in Poisson random environment. We show that, under a diffusive scaling involving different parameters of the system, the normalized point-to-point partition function of the polymer converges in law to the solution of the stochastic heat equation with fractional Laplacian and Gaussian multiplicative noise.

\subsection{The model}   
 Let $((X_t)_{t\geq 0}, \P_{x})$ denote the symmetric $\alpha$-stable process (with $\alpha \in (1,2]$) starting from $x\in \RR$ and denote $\P = \P_{0}$. The environment is a Poisson point process $\omega$ on $[0, \infty) \times \RR^d$ of intensity measure $vdsdx$. We assume that $\omega$ is defined on some probability space $(\Omega, \cG, \Q)$. Denote $\omega_t = \omega_{|[0,t]\times \RR^d}$ the restriction of the environment to times before $t \geq 0$. Fix $r > 0$, let $U(x)$ be the ball of volume $r^d$, centered at $x \in \RR^d$. Define $V_t(X)$ as the tube around path $X$:
$$ V_t(X) = \lacc (s,x): s \in (0, t], x \in U(X) \racc . $$
The polymer measure $\P_t^{\beta, \omega}$ is the Gibbsian probability measure on the space $\cD([0, \infty) \times \RR^d)$ of paths, defined by
$$ d\P_t^{\beta, \omega} = \frac{1}{Z_t(\omega, \beta, r)} \exp (\beta \omega(V_t)) d\P,$$
where $\beta \in \RR^d$ is the inverse temperature parameter, and $Z_t$ is the (point-to-line) partition function of the polymer:
$$Z_t(\omega, \beta, r) = \EE_{\P} [\exp (\beta \omega(V_t))].$$
For any fixed path $X$, we have $$ \EE_{\Q} [\exp (\beta \omega(V_t))]  = \sum_{k=0}^{\infty} e^{\beta k} e^{-v r^d t} \frac{(vr^d t)^k}{k!}  = \exp (\lambda(\beta)v r^d t),$$
 where $$ \lambda(\beta) = e^\beta - 1.$$
Hence, the partition function $Z_t$ has mean $\exp (\lambda(\beta)v r^d t)$.

The normalized partition function
$$ W_t(w, \beta, r) := e^{-\lambda(\beta)vr^d t} Z_t(\omega, \beta, r) $$
is a mean one, positive martingale. 
By Doob's martingale convergence theorem (see \cite[chapter 2, Corollary 2.11]{RevuzYor2005}), there exists a random variable $W_\infty$ such that
$$W_\infty  =\lim_{t \rightarrow \infty} W_t \qquad \text{a.s.} $$
In addition, from Kolmogorov's 0-1 law, we have a dichotomy: 
\begin{center}
    either \; $W_\infty = 0 $ \, a.s. \quad or \quad $ W_\infty > 0$ \, a.s. 
\end{center}
which correspond to strong disorder phase and weak disorder phase, respectively. In particular, we know from \cite{MiuraTawaraTsuchida2008} that when $d = 1$, $\alpha \in (1, 2]$ and $\beta \neq 0$, $$ W_\infty = 0, \, \, a.s. $$

To further describe the point-to-line partition function, we study the collection of paths started at $(0,0)$ and ended at $(t,x)$ through the point-to-point partition function.
 Let $p (t,x)$ denote the transition density of a symmetric $\alpha$-stable L\'evy process, whose characteristic function is $$ \int_{-\infty}^{+\infty} e^{ixz} p(t,x)dx = \exp \lpa - \nu t |z|^\alpha \rpa ,$$ where $\nu > 0$ is a fixed parameter.  
By self-similarity, the transition density satisfies the scaling transform
$$ t^{1/\alpha} p (tT, t^{1/\alpha} X ) = p(T,X). $$
We can now introduce the point-to-point partition function $$Z_{t,x} (\omega, \beta, r) = p(t,x) \EE_{\P} [\exp (\beta \omega(V_t)) | X_t = x] ,   $$ and the normalized point-to-point partition function: $$W_{t,x} (\omega, \beta, r) = p(t,x) e^{-\lambda(\beta)vr^d t} \EE_{\P} [\exp (\beta \omega(V_t)) |X_t = x ], $$
from which we can recover the point-to-line partition function
$$ Z_t (\omega, \beta, r) = \int_\RR Z_{t,x} (\omega, \beta, r)  dx,$$
and the normalized point-to-line partition function
$$ W_t (\omega, \beta, r) = \int_\RR W_{t,x} (\omega, \beta, r)  dx.$$

We stress that the random environment $\omega$ is contained in (normalized) point-to-line and point-to-point partition function, so that they are random variables on the probability space  $(\Omega, \cG, \Q)$. The partition functions are related to free energies in statistic mechanics, which contains a great amount of information. Many works have been done around the partition functions. In this paper, we will make further contributions in this direction.

\subsection{Main results}  
We will show similar results with those in Cosco \cite{COSCO2019805}. The following three assumptions and two theorems are the same as those in Cosco \cite{COSCO2019805} by letting $\alpha = 2.$ We prove here that it holds for $\alpha \in (1,2].$ 
\smallskip

\noindent
\textbf{Assumptions:} When $t\rightarrow \infty, $
\begin{itemize}
    \item[(a)] $ v_t r_t^2 \lambda(\beta_t)^2  \sim  (\beta^*)^2 \; t^{-(1-1/\alpha)}, \; \text{for some non-zero constant $\beta^*$},  $
    \item[(b)]  $ v_t r_t^{\alpha+1} \lambda(\beta_t)^{\alpha+1}  \rightarrow  0,$
    \item[(c)] $  \frac{r_t}{t^{1/\alpha}} \rightarrow  0. $ 
\end{itemize}

\begin{theorem}[Convergence of the normalized point-to-line partition function]
\label{main}
    Under assumptions (a)(b)(c), and as $t\rightarrow \infty$, 
    $$W_t (\omega^{v_t}, \beta_t, r_t)  \claw  \cZ_{\alpha, \beta^*},$$
    where $\omega^{v_t}$
is the Poisson point process with intensity measure $v_t dsdx$, $\cZ_{\alpha, \beta^*}$ is the random variable defined by (\ref{formulaZab}) in the next section.
\end{theorem}

\begin{theorem}[Convergence of the normalized point-to-point partition function]
\label{main2} Let $T > 0$ and $Y \in \RR$.
    Under assumptions (a)(b)(c), and as $t\rightarrow \infty$,
    $$ t^{1/\alpha} W_{tT, t^{1/\alpha} Y}(\omega^{v_{tT}}, \beta_{tT}, r_{tT})  \claw  \cZ_{\alpha, \beta^*/T}(T,Y),$$
    where $\cZ_{\alpha, \beta}(T,Y)$ is the random variable defined by (\ref{ EQsolutionSHElong} ) and (\ref{EQsolutionSHE}) in the next section.
\end{theorem}

%Define $D([0,1],\cD'(\cR)).$

%\begin{theorem}[Convergence in terms of processes]\label{main3} Define $$ \cY_t (T,Y) = p(T,Y) W_{tT, t^{1/\alpha} Y} (\omega^{v_t}, \beta_t, r_t).$$ Suppose that $(\beta_t)_{t\geq 0}$ is bound from above, then as $t\rightarrow \infty$: $$\cY_t (T,Y)  \claw  \cZ_{\alpha, \beta^*}(T,Y),$$ where the convergence in distribution holds in $D([0,1],\cD'(\cR)).$ \end{theorem}

 In the case that $\alpha = 2$, the strict positivity of $\cZ_{2, \beta}(T,Y)$ was proved by Mueller \cite{Mueller1991} in 1991 and a new proof was given by Flores \cite{Flores2014} in 2014. This theorem gained new attention due to the links among the stochastic heat  equation, the continuum directed polymer, and the KPZ equation (see \cite{AmirCorwinQuastel2011} for their relations, see \cite{Corwin2012}, \cite{Corwin2016}, \cite{JeremyKPZ2012} and \cite{Jeremy2015} for reviews on the KPZ universality class and the KPZ equation). If the solution of the stochastic heat  equation is strictly positive, then by the so-called Hopf-cole transformation, $ \ln \cZ_{\alpha, \beta}(T,Y)$ is the solution of a related KPZ equation. The strict positivity of 
$\cZ_{\alpha, \beta}(T,Y)$ for $\alpha \in (1,2)$ is an interesting open question, we left it for future research. 

Our proof is based on the theory of the Wiener integral and the Wiener--Itô integral with respect to Poisson process, especially the Wiener--Itô chaos expansion. 

The rest of the paper is organized as follows. In Section 2, we give some preliminaries, including the stochastic heat equation, the Wiener integral and the Wiener--Itô chaos expansion. We also state some useful propositions that will be used for the proofs of the main theorems. In Section 3, we prove our main results.

\section{Preliminaries}

\subsection{The stochastic heat equation}   
Consider the following stochastic heat equation:
\begin{equation}
\label{FSHE}
    \partial_t \cZ(t,x) = -\nu  (-\Delta)^{\alpha/2} \cZ(t,x)  + \beta \cZ(t,x) \xi(t,x),   \quad \cZ(0,x) = \delta_0(x),
\end{equation}
where $\beta > 0$ is a constant, and $\xi$ is a Gaussian noise with the covariance 
\begin{equation}
    \text{Cov} (\xi(t,x), \xi(s,y)) = \delta_0(t-s) \delta_0(x-y).
\end{equation}
A space-time white noise can be realized as the distributional mixed derivative of a two-sided Brownian sheet $W := \lacc W(t,x) \racc _{t\geq 0, x \in \RR}$, i.e.
$$ \xi (t,x) = \frac{\partial^2 W(t,x)}{\partial t \partial x}.$$
Here, $W$ is a mean-zero Gaussian process whose covariance function  is 
$$  \text{Cov} (W(t,x), W(s,y)) = \min (s,t)\times \min(|x|,|y|) \times {\bf 1}_{(0, \infty)}(xy), $$
for all $s, t > 0$ and $x, y \in \RR$.

Wiener integral uses the white noise to map elements of $L^2([0,t]\times \RR)$ into Gaussian random variables in a way that preserves the Hilbert space structure of the $L^2$ space. For $g \in L^2([0,t]\times \RR)$, we fix the notation
\begin{equation*}
    I_{1,t}(g) := \int_{[0,t] \times \RR} g(s, x) \xi(s, x)dsdx.
\end{equation*}
We refer to \cite{Janson1997gaussian} for a more detailed definition. Multiple Wiener integrals $I_{k,t}$, for $ k > 1, t>0,$ can be defined as 
\begin{equation*}
    I_{k,t}(g) := \int_{\Delta_k(t) \times \RR^k} g(\s, \x) \xi^{\otimes k }(\s, \x) d\s d\x,
\end{equation*}
for all $g \in L^2(\Delta_k(t) \times \RR^k),$ 
where $\s = (s_1, \cdots, s_k), \x =  (x_1, \cdots, x_k)$, $s_1, \cdots, s_k \in \RR_+,   x_1, \cdots, x_k \in \RR, $ and 
$ \Delta_k(t) = \lacc \s \in [0,t]^k \, | \, 0 < s_1 < \cdots < s_k \leq t  \racc.$

It is known that $(s,t,y,x) \rightarrow p(t - s, x -y) $ is the fundamental solution to the heat operator $(\partial/\partial t) -  ( -\nu  (-\Delta)^{\alpha/2}) $. According to
the theory of Walsh \cite{Walsh1986}, we may write equation (\ref{FSHE}) in mild form as follows
\begin{equation}
\label{DuhamelForm}
    \cZ (t,x) = \int_{\RR} p(t, x-y) \cZ(0, y)dy + \beta \int_0^t \int_\RR p(t-s, x-y) \cZ (s,y )\xi(s,y)dyds.
\end{equation}
This progressively measurable process $\cZ(t,x)$ is called a mild solution of equation (\ref{FSHE}) if it is well-defined.
By replacing $\cZ(0, y)$ with $\delta_0(y)$ and iterating (\ref{DuhamelForm}) multiple times,  we derive a Wiener chaos expansion of $\cZ(t,x)$, written by
\begin{equation}
\label{ EQsolutionSHElong}
\cZ_{\alpha, \beta} (t,x) = p(t,x) + \sum_{k = 1}^{\infty} \beta^k 
\int_{\Delta_k(t)} \int_{\RR^k} p(t-s_k, x - x_k) \prod_{i=1}^k  p(s_i - s_{i-1}, x_i - x_{i-1})\xi(s_i, x_i) dx_i ds_i,
 \end{equation}
where $ s_0 = 0$ and $ \Delta_k(t) $ is defined as above.
In order to use a more succinct notation, for $k > 0$ we  denote by $p^k$ the multi-step $\alpha$-stable transition function
$$p^k (\s,\x) = \prod_{j=1}^{k} p (s_{j}-s_{j-1}, x_{j} - x_{j-1}), $$
and for fixed $t \in \RR_+$ and $ x \in \RR$ we define
 $$ p^k(\s,\x; t,x) =  p^k(\s,\x)p(t-s_k, x - x_k), $$
where  $(\s, \x) \in \Delta_k(t) \times \RR^k$ and $(s_0, x_0) = {\bf 0}$.

\begin{proposition}
    \label{SolutionSHE}
    If $\cZ(0,x) = \delta_0(x)$, then there exists a unique mild solution to equation (\ref{FSHE}) with initial data $\delta_0(x)$. The solution can be written by 
    \begin{equation}
    \label{EQsolutionSHE}
 \cZ_{\alpha, \beta}(t,x) = \sum_{k = 0}^{\infty}  I_{k,t}( \beta^k p^k(\cdot,\cdot; t,x) ),
    \end{equation}
where $p^0(\cdot,\cdot; t,x) = p(t,x)$ and $I_{0,t}(f) = f.$
%Besides, there exists a positive constant $C$ such that, for any $t \in [0, T]$, $$\EE_\Q [(\cZ(t,x) )^2] \leq C p(t,x)^2. $$
\end{proposition}

\proof
The uniqueness and existence of mild solutions for equation (\ref{FSHE}) with $\alpha \in (1,2]$ was proved as a special case in \cite[Section 4]{CJKS14}.
We use the Picard iteration to find the explicit expression of the solution. Let $\cZ_0(t,x) \equiv 0.$ For $k \geq 0,$
define 
$$ \cZ_{k+1} (t,x) = \int_{\RR} p(t, x-y) \cZ(0, y)dy + \beta \int_0^t \int_\RR p(t-s, x-y) \cZ_k (s,y )\xi(s,y)dyds, $$
and 
$$ \bar{\cZ}_{k} (t,x) = \cZ_{k+1} (t,x) - \cZ_{k} (t,x).$$
Then these processes are progressively measurable by construction, $\bar{\cZ}_{0} (t,x) = p(t,x), $ and for $k \geq 1$
$$ \bar{\cZ}_{k} (t,x) = \beta^k 
\int_{\Delta_k(t)} \int_{\RR^k} p(t-s_k, x - x_k) \prod_{i=1}^k  p(s_i - s_{i-1}, x_i - x_{i-1})\xi(s_i, x_i) dx_i ds_i,$$
from which we have
\begin{align*}
    \EE [(\bar{\cZ}_k (t,x) )^2]   &  = \beta^{2k} \int_{\Delta_k(t)} \int_{\RR^k} p(t-s_k, x - x_k)^2 \prod_{i=1}^k  p(s_i - s_{i-1}, x_i - x_{i-1})^2 dx_i ds_i \\
     & \leq  \beta^{2k} M_\alpha^{k+1} \int_{\Delta_k(t)}  (t-s_k)^{-1/\alpha} \prod_{i=1}^k  (s_i - s_{i-1})^{-1/\alpha} ds_i     \\
     & = \beta^{2k} M_\alpha^{k+1} t^{k(1-1/\alpha) - 1/\alpha} \frac{(\Gamma(1-1/\alpha))^{k+1}}{\Gamma((k+1)(1-1/\alpha))}, 
\end{align*}
where $M_\alpha := \text{max}_{x \in \RR} ( p (1,x))$. The last equality follows from the Beta integral formula
 \begin{equation*}
        \int_{t_i \geq 0, \, t_1 + \cdots + t_n \leq 1} \lpa 1 - \sum_{i=1}^n t_i \rpa ^{\beta -1} \prod_{j = 1}^n t_j^{\alpha_j - 1}dt_j = \frac{\Gamma(\beta)\prod_{j = 1}^n \Gamma(\alpha_j)}{\Gamma( \beta + \sum_{i=1}^n \alpha_i) }. 
    \end{equation*}
Therefore, 
$$  \EE [(\cZ_{\alpha, \beta}(t,x) )^2] < \infty, $$
the random variable in the right-hand side of formula (\ref{ EQsolutionSHElong}) is well-defined, and it is a mild solution to equation (\ref{FSHE}). 
\endproof

\begin{proposition} 
\label{Zalphabeta}
For every $\alpha \in (1, 2]$, the random variable 
\begin{equation}
    \label{formulaZab}
    \cZ_{\alpha, \beta} := \sum_{k = 0}^\infty I_k (\beta^k p^k) = 1 +  \sum_{k = 1}^{\infty} \beta^k 
\int_{\Delta_k(1)} \int_{\RR^k} \prod_{i=1}^k  p(s_i - s_{i-1}, x_i - x_{i-1})\xi(s_i, x_i) dx_i ds_i
\end{equation}
is well defined and square-integrable, where $p^0 = 1$ and $I_k := I_{k,1}.$
\end{proposition}
%$\cZ_{\alpha, \beta} $ is the point-to-line partition function at time $t = 1$. 
We remark that the random variable above is related to $\cZ_{\alpha, \beta}(t,x) $ defined in (\ref{EQsolutionSHE}) via
$$ \cZ_{\alpha, \beta} = \int_\RR \cZ_{\alpha, \beta}(1,x)dx.  $$

\proof 
We recall the covariance relation of multiple Wiener integral. For $g \in L^2(\Delta_k(1)\times \RR^k)$ and $ h \in L^2(\Delta_j(1)\times \RR^j)$,
$$ \EE \lcr  I_k(g)I_j(h) \rcr = {\bf 1}_{\lacc k=j \racc } \int_{\Delta_k(1)\times \RR^k} gh d\s d\x. $$
It implies that
$$\EE[(\cZ_{\alpha, \beta} )^2] =  \sum_{k = 0}^\infty \EE\lcr \lpa I_k (\beta^k p^k) \rpa^2  \rcr = \sum_{k = 0}^\infty || \beta^k p^k ||_{\Delta_k(1)\times \RR^k}^2. $$
To prove that $\cZ_{\alpha, \beta}$ is square-integrable, it suffices to verify that $\sum_{k = 0}^\infty || \beta^k p^k ||_{\Delta_k(1)\times \RR^k}^2 < \infty$. We bound this infinite sum as follows:
\begin{align*}
      & \int_{\Delta_k(1)}\int_{\RR^k} p^k(\s, \x)^2 d\s d\x \\
    = & \int_{\Delta_k(1)}\int_{\RR^k} \lpa \prod_{j=1}^{k} p (s_{j}-s_{j-1}, x_{j} - x_{j-1}) \rpa^2 d\s d\x  \\
    = & \int_{\Delta_k(1)}\int_{\RR^k} \lpa \prod_{j=1}^{k} (s_{j}-s_{j-1})^{-\frac{1}{\alpha}} p (1, 
 (s_{j}-s_{j-1})^{-\frac{1}{\alpha}}(x_{j} - x_{j-1})) \rpa^2 d\s d\x  \\
    \leq &  M_\alpha^k \int_{\Delta_k(1)}\int_{\RR^k} \prod_{j=1}^{k} (s_{j}-s_{j-1})^{-\frac{2}{\alpha}} p (1, 
 (s_{j}-s_{j-1})^{-\frac{1}{\alpha}}(x_{j} - x_{j-1}))  d\s d\x  \\
    = &  M_\alpha^k \int_{\Delta_k(1)} \prod_{j=1}^{k} (s_{j}-s_{j-1})^{-\frac{1}{\alpha}}d\s  \\
     = &  \big( M_\alpha \Gamma(1-1/\alpha)  \big) ^k \Big( \Gamma(k - \frac{k}{\alpha} + 1)\Big) ^{-1}, 
\end{align*}
where $M_\alpha := \text{max}_{x \in \RR} ( p (1,x))$ and the last equality follows from the Beta integral formula as in the proof of Proposition \ref{SolutionSHE}.
Hence,
\[ \EE [(\cZ_{\alpha, \beta} )^2] = \sum_{k=0}^\infty \beta^k \int_{\Delta_k(1)}\int_{\RR^k} p^k(\s, \x)^2 d\s d\x \leq \sum_{k=0}^\infty    \big(\beta M_\alpha \Gamma(1-\frac{1}{\alpha})  \big)^k   \Big(\Gamma(k - \frac{k}{\alpha} + 1)  \Big)^{-1} < \infty. \]
\endproof

\subsection{The Wiener--Itô chaos expansion}
For the basic theory of the Wiener--Itô integrals with respect to Poisson process, the readers are recommended to read \cite{LastPenrose2017lectures}. Here, we give a rough introduction of the chaos expansion. We start with the definitions of factorial measures and the Wiener--Itô integrals, by means of which we can formulate the Wiener--Itô chaos expansion.

\begin{definition}[Section 4.2 in \cite{LastPenrose2017lectures}]
Let $k \in \NN_0 \cup \lacc \infty \racc$ and $\omega = \sum_{j=1}^k \delta_{x_j}$ be a point process on an arbitrary measurable space $(S, \cS)$. We define 
 \begin{equation*}
     \omega^{(m)}(C) = \sum_{i_1,\cdots,i_m \leq k}^{\neq} {\bf 1}_{\lacc (x_{i_1}, \cdots, x_{i_m}) \in C \racc }, \quad C \in \cS^m,
 \end{equation*}
 where the superscript $\neq$ indicates summation over $m$-tuples with pairwise different entries. The measure
 $\omega^{(m)}$ is called the $m$-th \textbf{factorial measure} of $\omega$.
\end{definition}

\begin{definition}[Definition 12.10 in \cite{LastPenrose2017lectures}] Let $\omega$ be a Poisson point process with intensity measure $v$. For $m \in \NN$ and symmetric $f \in L^2(v^m)$, the random variable
\begin{equation}
    \bar{\omega}^{(m)} (f) = \sum_{k=0}^m (-1)^{m-k} \binom{m}{k} \omega^{(k)} \otimes v^{m-k} (f)
\end{equation}
is called the (stochastic, m-th order) \textbf{Wiener--Itô integral} of f. When $m=0$, define $\bar{\omega}^{(0)} (f)$ to be the identity, i.e. $\bar{\omega}^{(0)} (f) := f.$
\end{definition}

The most important fact of the higher-order 
Wiener--Itô integrals that we will use later is the covariance relation.

\begin{proposition}[Corollary 12.8 in \cite{LastPenrose2017lectures}]
\label{Covariancerelation}
Let $m,n \in \NN$, $f \in L^2(v^m)$ and $g \in L^2(v^n )$. Assume that both $f$ and $g$ are symmetric, bounded and satisfy 
 $v^m(\lacc f \neq 0 \racc) < \infty$ and $v^n(\lacc g \neq 0 \racc) < \infty$. Then 
$$ \EE [ \bar{\omega}^{(m)} (f) \bar{\omega}^{(n)} (g)] = {\bf 1}_{\lacc m=n \racc} m! \int fg d v^m . $$
\end{proposition}

\begin{proposition}[Theorem 18.10 in \cite{LastPenrose2017lectures}] 
\label{chaosexpansion}
For a Poisson point process $\omega$, if $\EE[f(\omega)^2] < \infty$, then $f(\omega)$ has the following \textbf{chaos expansion}
\begin{equation}
    f(\omega) = \sum_{k=0}^{\infty} \frac{1}{k!}\bar{\omega}^{(k)} (T_k f),
\end{equation}
where 
$$ T_k f(x_1, \cdots, x_k) := \EE \lcr D_{x_1, \cdots, x_k}^k f(\omega) \rcr$$
is symmetric and square-integrable, and
 $D_{x_1, \cdots, x_k}^k$ is the difference operator, defined by the following iteration:
 \begin{align*}
     D^0 f &:= f, \\
     D^1_x f (\omega) = D_x f(\omega) &:= f(\omega + \delta_x) - f(\omega), \\
      D_{x_1, \cdots, x_k}^k f(\omega) &:= D_{x_1}^1 D_{x_2, \cdots, x_k}^{k-1} f(\omega).
 \end{align*}
\end{proposition}

Proposition \ref{chaosexpansion} states that
every square-integrable function
admits a Wiener–Itô chaos expansion, that is, it can be written as an infinite sum of orthogonal
multiple Wiener–Itô integrals. The normalized partition function is square-integrable (see \cite{MiuraTawaraTsuchida2008}), so it admits a Wiener–Itô chaos expansion. We mimic the proof of \cite[Proposition 3.2]{COSCO2019805} and obtain an explicit expression. 

\begin{proposition}
\label{WienerIto}
The normalized partition function has the Wiener--Itô chaos expansion
$$ W_t = \sum_{k=0}^{\infty} \frac{1}{k!}\bar{\omega}^{(k)}_t (T_k W_t), $$
where, for all $\s \in [0,t]^k, \x \in \RR^k$ and $k \geq 0$,
$$T_k W_t (\s, \x) = \lambda(\beta)^k \EE_{\P}\lcr \prod_{i = 1}^{k} \chi_{s_i,x_i}^{r}(X)\rcr$$
and $\chi_{s,x}^{r}(X) = {\bf 1}_{\lacc |x - X_s| \leq r/2 \racc } $, with the convention that an empty product equals 1.
\end{proposition}

\begin{proof} 
For a fixed path $X$,
\begin{equation*}
    \omega(V_t(X)) = \int_{[0,t]\times \RR}  \chi_{s,y}^{r}(X)\omega(dsdy).
\end{equation*}
For all $(s,x) \in [0,t]\times \RR,$ we have (for more details, see \cite[section 18.1]{LastPenrose2017lectures})
\begin{align*}
    D_{s,x} e^{\beta \omega(V_t)} &= e^{\beta \int_{[0,t]\times \RR}  \chi_{u,y}^{r}(X)(\omega+ \delta_{s,x})(dudy) } -  e^{\beta \omega(V_t)} \\
    & = e^{\beta \omega(V_t)} ( e^{\beta \chi_{s,x}^{r}(X) } -1 ) \\
    &= e^{\beta \omega(V_t)} \lambda(\beta)\chi_{s,x}^{r}(X).
\end{align*}
Then we deduce that
\begin{align*}
       T_k W_t (\s, \x)  = \EE_\Q \EE_\P \lcr e^{\beta \omega(V_t)}\prod_{i = 1}^{k} \lambda(\beta)\chi_{s_i,x_i}^{r}(X)\rcr e^{-tv\lambda r^d}  
           = \lambda(\beta)^k \EE_{\P}\lcr \prod_{i = 1}^{k} \chi_{s_i,x_i}^{r}(X)\rcr. 
\end{align*}
\end{proof}

\section{Proofs of main theorems}
Inspired by \cite[Section 6.2]{COSCO2019805}, we introduce the generalized time-depending functions
\begin{equation}
\label{Defphi}
  \phi_t^k(\s,\x) = \gamma_t^{-k}\lambda(\beta_t)^k \EE_{\P}\lcr \prod_{i = 1}^{k} \chi_{s_i,x_i}^{r_t/t^{1/\alpha}}(X)\rcr {\bf 1}_{\Delta_k(1)}(\s),  
\end{equation}
which extends Cosco's definition for $\alpha = 2$ to arbitrary $\alpha \in (1,2]$. Here,
\begin{equation}
    \label{gamma_t}
    \gamma_t : = (\beta^*)^{-(\alpha+1)/(\alpha-1)} v_t^{1/(\alpha-1)} r_t^{(\alpha+1)/(\alpha-1)} \lambda(\beta_t)^{(\alpha+1)/(\alpha-1)}.
\end{equation}
When $\alpha = 2$, it coincides with \cite[formula (64)]{COSCO2019805}. The diffusive scaling property of stable processes implies that
\begin{equation}
    \label{formula_scaling}
    \chi_{s/t,x/t^{1/\alpha}}^{r_t/{t^{1/\alpha}}} (X)\elaw  \chi_{s,x}^{r_t} (X). 
\end{equation}
Combining Proposition \ref{WienerIto} with \eqref{Defphi}, we find that $$W_t = \sum_{k=0}^{\infty} \gamma_t^k \bar{\omega}_t^{(k)} (\tilde{\phi}_t^k),$$
where  $$\tilde{\phi}_t^k :=  \phi_t^k (\cdot / t, \cdot / t^{1/\alpha}). $$
For any function $g$ defined on $\Delta_k(1) \times \RR^k$, let $ \tilde{g}_t$ denote the rescaled function of $g$, which is defined on $\Delta_k(t) \times \RR^k$ as
$$\tilde{g}_t (\s,\x) :=  g (\s / t, \x / t^{1/\alpha}). $$

\begin{proposition}
    \label{k-thItemConvergence}
    Let $g \in L^2(\Delta_k(1) \times \RR^k)$
    for $k \geq 1$.  As $t \rightarrow \infty,$
    \begin{equation}
        \label{formula-k-thConve}
        \gamma_t^k \bar{\omega}_t^{(k)} (\tilde{g}_t) \claw  I_k(g).
    \end{equation}
\end{proposition}

In order to prove Proposition \ref{k-thItemConvergence}, we need the Poisson formula for exponential moments and two lemmas below. 

\begin{itemize}
    \item For all non-negative and all non-positive measurable functions $h$ on $\RR_+ \times \RR^d,$ the \textit{Poisson formula for exponential moments} (see \cite[Chapter 3]{LastPenrose2017lectures})  writes
\begin{equation*}
     \EE_\Q  \lcr  e^{\int h(s,x)w_t(dsdx)}  \rcr  = \exp \int_{]0,t] \times \RR} (e^{h(s,x)} - 1 ) vdsdx.
\end{equation*}
    The formula remains true when $h$ is replaced by $ih$, for any real integrable function $h$.
\end{itemize}

\begin{lemma}
    \label{IneqExp}
    For any $z \in \CC$ with $\mathrm{Re}(z) \leq $ 0,
    $$ \vert e^z - \sum_{k=0}^n  \frac{1}{k!} z^k \vert \leq \frac{1}{(n+1)!}  \vert z \vert ^{n+1}.$$
\end{lemma}
This lemma is easy to prove by induction, we omit the proof.

\begin{lemma}
\label{diagram}
    Let $(S, \cS)$ be a metric space with its Borel $\sigma$-field, and $(X_t^n, Y_t)$ be random variables on $S^2$ for $t \geq 0, n \in \NN$.  If the diagram 
\begin{align*}
                    & X_t^n      &     &\underset{t \rightarrow \infty }{\overset{(d)}{\xrightarrow{\hspace*{3cm}}}} &     & Y^n \\
  \PP, \mathrm{unif \, in \,} t  \, & \Bigg\downarrow \, n \rightarrow \infty &  &   &  (d)  &\Bigg\downarrow  n \rightarrow \infty  \\
  & Y_t   &      &  &  &Y
\end{align*}
holds, then $Y_t \claw Y.$
    
\end{lemma}

For a proof of this lemma, see \cite[Chapter 1, Theorem 3.2]{Billingsley1999}.  

\begin{proof}[Proof of Proposition \ref{k-thItemConvergence}]
We first consider the case of $k=1$. We compute the characteristic function of $\gamma_t \bar{\omega}_t (\tilde{g}_t).$ Let $u \in \RR$, we have
\begin{equation*}
    \EE_\Q  \lcr  e^{iu\gamma_t \bar{\omega}_t (\tilde{g}_t)}  \rcr  = \exp \lpa \int_{[0,1]} \int_{\RR} t^{1+1/\alpha} ( e^{iu\gamma_t g(s,x) } - 1 -  iu\gamma_t g(s,x)  )  v_t dsdx \rpa .
\end{equation*}
By Lemma \ref{IneqExp}, we have 
$$ | t^{1+1/\alpha} ( e^{iu\gamma_t g(s,x) } - 1 -  iu\gamma_t g(s,x)  )  v_t | \leq t^{1+1/\alpha} v_t \gamma_t^2 \frac{u^2}{2} g(s,x)^2.$$
Since $g$ is square-integrable and assumptions (a)-(b) imply that $v_t \gamma_t^2 \sim t^{-(1+1/\alpha)} $, we can use the dominated convergence theorem. Assumption (b) implies that $\gamma_t \rightarrow 0$. Using again the asymptotic equivalence  $v_t \gamma_t^2 \sim t^{-(1+1/\alpha)} $, we conclude that the integrand converges pointwise to $-\frac{u^2}{2} g^2(s,x)$, therefore, as $t \rightarrow \infty,$
$$\EE_\Q  \lcr  e^{iu\gamma_t \bar{\omega}_t (\tilde{g}_t)}  \rcr \, \rightarrow \, \exp \lpa -||g||_2^2 u^2 /2 \rpa.$$
The limit is the Fourier transform of a centered Gaussian random variable of variance $||g||_2^2$, which has the same law as $I_1(g).$

We then consider the case of $k>1$.  Let $g^1, \cdots, g^k$ be the functions whose supports are disjoint, finite and measurable sets $A_1, \cdots, A_k \subset [0,1]\times \RR$, and consider 
\begin{equation}
    \label{prodctg}
    g(\s, \x) = g^1(s_1,x_1) \cdots g^k(s_k,x_k).
\end{equation}
The property of higher order Wiener--Itô integral (see \cite[Exercise 12.8]{LastPenrose2017lectures}) yields
$$ \gamma_t^k \bar{\omega}_t^{(k)} (\tilde{g}_t)  =  \prod_{i=1}^k \gamma_t \bar{\omega}_t (\tilde{g}_t^i). $$
From the $k = 1$ case, for every $i$, we have $ 
 \gamma_t \bar{\omega}_t (\tilde{g}_t^i) \claw I_1(g^i)$, as $t \rightarrow \infty$, which implies that
$$ \gamma_t^k \bar{\omega}_t^{(k)} (\tilde{g}_t)  \claw  \prod_{i=1}^k I_1(g^i) = I_k(g), \quad t \rightarrow \infty,$$
where the equality holds since the functions $g^i$ are orthogonal in $L^2$.

Denote by $V_k$ the linear subspace of $L^2(\Delta_k(1) \times \RR^k)$ spanned by the functions of the form (\ref{prodctg}). By the linearity of $\bar{\omega}_t^{(k)}$, the convergence in distribution
\begin{equation}
        \label{formula-k-thConveInV_k}
        \gamma_t^k \bar{\omega}_t^{(k)} (\tilde{g}_t) \claw  I_k(g)
    \end{equation}
is valid for every function $g$ in $V_k$.

Note that $V_k$ is dense in $L^2(\Delta_k(1) \times \RR^k)$. Then for any $g$ in $L^2(\Delta_k(1) \times \RR^k)$, there exists a sequence of functions $(g^n)_{n \geq 1}$ in $V_k$ that converges to $g$ in $L^2$ norm. 
We want to prove the following diagram:

\begin{align*}
                    & \gamma_t^k \bar{\omega}_t^{(k)} (\tilde{g}_t^n)      &     &\underset{t \rightarrow \infty }{\overset{(d)}{\xrightarrow{\hspace*{3cm}}}} &     & I_k (g^n) \\
  L^2, \text{unif in t}  \, & \Bigg\downarrow \, n \rightarrow \infty &  &   &  L^2 &\Bigg\downarrow   n \rightarrow \infty \\
  & \gamma_t^k \bar{\omega}_t^{(k)} (\tilde{g}_t)   &      &  &  &I_k(g)
\end{align*}
Using again the asymptotic equivalence  $v_t \gamma_t^2 \sim t^{-(1+1/\alpha)} $ and Proposition \ref{Covariancerelation}, there exists a constant $C$ such that 
\begin{align*}
& \EE_\Q \lcr (\gamma_t^k \bar{\omega}_t^{(k)} (\tilde{g}_t^n)  - \gamma_t^k \bar{\omega}_t^{(k)} (\tilde{g}_t))^2  \rcr  \\
=& v_t^k \gamma_t^{2k} ||(g^n-g)(\cdot/t, \cdot/t^{1/\alpha})||_{L^2(\Delta_k(1) \times \RR^k)}^2 \\
=&  v_t^k \gamma_t^{2k}  t^{k(1+1/\alpha)}||g^n-g||_{L^2(\Delta_k(1) \times \RR^k)}^2  \\
\leq& C ||g^n-g||_{L^2(\Delta_k(1) \times \RR^k)}^2 \\
\rightarrow& 0,  \qquad \text{as}  \quad n \rightarrow \infty.
\end{align*}
Similarly, 
$$  \EE \lcr ( I_k (g^n) - I_k (g) )^2   \rcr  = ||g^n-g||_{L^2(\Delta_k(1) \times \RR^k)}^2\rightarrow 0, $$
as $n \rightarrow \infty.$
We then conclude the proof by Lemma \ref{diagram}.
\end{proof}

\begin{corollary}[Corollary 5.2 in \cite{COSCO2019805}]
\label{CoroConvergence}
Let $(g^k)_{k\geq 0}$ be a family of functions defined on $\Delta_k(1)\times \RR^k$. If
 $\sum_{k = 0}^\infty || g^k ||_{\Delta_k(1)\times \RR^k}^2 < \infty$,
then the sum $ \sum_{k=0}^{\infty}  \gamma_t^k \bar{\omega}_t^{(k)} (\tilde{g}_t^k)$ is well defined, and as 
    $t \rightarrow \infty,$
$$ \sum_{k=0}^{\infty}  \gamma_t^k \bar{\omega}_t^{(k)} (\tilde{g}_t^k) \claw \sum_{k=0}^{\infty}  I_k(g^k).$$
\end{corollary}

%By Proposition 4.2 in \cite{COSCO2019805}, we have $$ \sum_{k=0}^{M}  I_k(g^k) \rightarrow \sum_{k=0}^{\infty}  I_k(g^k), \quad \text{in $L^2$, as } M \rightarrow \infty.  $$

%The proof of Corollary \ref{CoroConvergence} is similar with the proof of Corollary 5.2 in \cite{COSCO2019805}, except replacing $v_t \gamma_t^2 \sim t^{-3/2}  $ with $v_t \gamma_t^2 \sim t^{-(1+1/\alpha)} .$

\begin{lemma}
\label{lemmabound}
Let $k$ be a positive integer.
\begin{itemize}
    \item[(1)] There exists a non-negative function $H_k \in L^1(\Delta_k(1))$ such that 
    $$ \int_{\RR^k} \lpa \epsilon^{-k} \EE_{\P}\lcr \prod_{i = 1}^{k} 
\chi_{s_i,x_i}^{\epsilon}(X)\rcr \rpa ^2 d\x \leq H_k(\s), \quad \forall \, \epsilon > 0, \forall \, \s \in \Delta_k(1).$$
    \item[(2)] As $\epsilon \rightarrow 0$, we have the pointwise convergence:
    $$ \epsilon^{-k} \EE_{\P}\lcr \prod_{i = 1}^{k} 
\chi_{s_i,x_i}^{\epsilon}(X)\rcr \rightarrow p^k(\s, \x), \quad \forall \, \s \in \Delta_k(1), \forall \, \x \in \RR^k.$$
\end{itemize}
\end{lemma}

\proof  (1) By the Markov property of stable processes, 
\begin{align}
\label{Markovproperty}
  \nonumber    &  \epsilon^{-k} \EE_{\P}\lcr \prod_{i = 1}^{k} 
\chi_{s_i,x_i}^{\epsilon}(X)\rcr \\
   \nonumber   = & \epsilon^{-k} \int_{\RR^k} \prod_{i=1}^k {\bf 1}_{|x_i - y_i| \leq \epsilon/2} \,\, p(s_i-s_{i-1}, y_i - y_{i-1}) d\y   \\
     \vspace{5mm}
      = &  \int_{[-\frac{1}{2}, \frac{1}{2}]^k} \prod_{i=1}^k  p(s_i-s_{i-1}, x_i - x_{i-1} + \epsilon(u_i - u_{i-1})) d\u.  
\end{align}

By H\"older's inequality, 
\begin{align*}
      & \int_{\RR^k} \lpa \epsilon^{-k} \EE_{\P}\lcr \prod_{i = 1}^{k} 
\chi_{s_i,x_i}^{\epsilon}(X)\rcr \rpa ^2 d\x \\
     \leq  & \int_{\RR^k} \int_{[-\frac{1}{2}, \frac{1}{2}]^k} \prod_{i=1}^k  p(s_i-s_{i-1}, x_i - x_{i-1} + \epsilon(u_i - u_{i-1}))^2 d\u  d\x  \\
      =  & \int_{\RR^k} \prod_{i=1}^k  p(s_i-s_{i-1}, z_i - z_{i-1} )^2 d\z   \\
    \leq &  M_\alpha^k  \int_{\RR^k} \prod_{i=1}^{k} (s_{i}-s_{i-1})^{-\frac{2}{\alpha}} p (1, 
 (s_{i}-s_{i-1})^{-\frac{1}{\alpha}}(z_{i} - z_{i-1}))   d\z  \\
    = &  M_\alpha^k  \prod_{i=1}^{k} (s_{i}-s_{i-1})^{-\frac{1}{\alpha}}, 
\end{align*}
where we have done a change of variable $z_i = x_i + \epsilon u_i$.

 We define $$ H_k(\s) = M_\alpha^k  \prod_{i=1}^{k} (s_{i}-s_{i-1})^{-\frac{1}{\alpha}}.$$
We saw from the proof of Proposition \ref{Zalphabeta} that $H_k$ is in $ L^1(\Delta_k(1))$. Thus $H_k$ is a desired function.

(2) This result follows immediately from equality (\ref{Markovproperty}).
\endproof 

\begin{lemma}
\label{ConvergenceLemma}
Let $k$ be a positive integer.
\begin{itemize}
    \item[(1)] As $t \rightarrow \infty$,
    $$ ||\phi_t^k  - (\beta^*)^k p^k||_{L^2(\Delta_k(1) \times \RR^k)} \rightarrow 0.$$ 
    \item[(2)] There exists a positive constant $C(\beta^*)$ such that 
    $$ \sup_{t\in [0,\infty)} ||\phi_t^k||_{L^2([0,1]^k\times \RR^k)} \leq  C(\beta^*)^k  \big(M_\alpha \Gamma(1-\frac{1}{\alpha})  \big)^k   \Big(\Gamma(k - \frac{k}{\alpha} + 1)  \Big)^{-1} .$$
\end{itemize} 
\end{lemma}

\proof
(1)
Assumptions (a) and (b) imply that
\begin{equation}
\label{ConvLemma2}
    \gamma_t^{-1}\lambda(\beta_t) \; \sim \; \beta^* \, t^{1/\alpha}/ r_t , \quad  t \rightarrow \infty.
\end{equation}
Combining Assumption (c) and Lemma \ref{lemmabound} (2), we have the pointwise convergence
$$  \phi_t^k(\s,\x) \rightarrow  (\beta^*)^k p^k(\s, \x),  \quad \text{as} \;\, t \rightarrow \infty, $$
for all $\s \in \Delta_k(1)$ and all $ \x \in \RR^k.$
A simple computation shows 
$$ \int_{\RR^k} \lpa \phi_t^k(\s,\x) - (\beta^*)^k p^k(\s, \x) \rpa^2 d\x \leq \int_{\RR^k} 2 \phi_t^k(\s,\x)^2 +2 (\beta^*)^{2k} p^k(\s, \x)^2 d\x.$$

By Lemma \ref{lemmabound} (1) and the dominated convergence theorem, we obtain the convergence
$$ ||\phi_t^k  - (\beta^*)^k p^k||_{L^2(\Delta_k(1) \times \RR^k)}^2  = \int_{\Delta_k(1)} \int_{\RR^k} \lpa \phi_t^k(\s,\x) - (\beta^*)^k p^k(\s, \x) \rpa^2 d\x d\s \rightarrow 0,   \quad t \rightarrow \infty .$$

(2) By the asymptotic equivalence (\ref{ConvLemma2}) and Lemma \ref{lemmabound} (1), there exists a constant $C(\beta^*)$ such that
$$ ||\phi_t^k||_{L^2([0,1]^k\times \RR^k)} \leq C(\beta^*)^k \int_{\Delta_k(1)} H_k (\s) d\s.$$
By the same calculation in the proof of Proposition \ref{Zalphabeta}, we then finish the proof.
\endproof

\subsection{Proof of Theorem \ref{main}} 
\proof 
We have seen from the proof of Proposition \ref{Zalphabeta} that 
 $\sum_{k = 0}^\infty || (\beta^*)^k p^k ||_{\Delta_k(1)\times \RR^k}^2 < \infty$ . By Corollary \ref{CoroConvergence}, as $t \rightarrow \infty,$
$$ \sum_{k=0}^{\infty} (\beta^*)^k \gamma_t^k \bar{\omega}_t^{(k)}( \tilde{p}^k_t )\claw \sum_{k=0}^{\infty}  I_k((\beta^*)^k p^k)  = \cZ_{\alpha, \beta^*} . $$
It is known that, for real random variables $X_n$ and $Y_n$, if $X_n \claw X$ and $||X_n - Y_n||_2 \rightarrow 0$, then $Y_n \claw X $ (see \cite[Theorem 3.1]{Billingsley1999}). 
Hence, it remains to prove that 
\begin{equation}
\label{EqThm1}
    || \sum_{k=0}^{\infty}  \gamma_t^k \bar{\omega}_t^{(k)}( \tilde{\phi}^k_t ) - \sum_{k=0}^{\infty} (\beta^*)^k \gamma_t^k \bar{\omega}_t^{(k)}( \tilde{p}^k_t ) ||_2^2  \rightarrow 0,  \, t \rightarrow \infty. 
\end{equation}
By the linearity of $\bar{\omega}_t^{(k)}$ and Proposition \ref{Covariancerelation},  the left side can be written as
$$ \sum_{k=0}^{\infty}  \gamma_t^{2k} ||\bar{\omega}_t^{(k)}( \tilde{\phi}^k_t  -  (\beta^*)^k \tilde{p}^k_t) ||_2^2 =   \sum_{k=0}^{\infty}  \gamma_t^{2k}v_t^k t^{k(1+1/\alpha)} ||\phi^k_t  -  (\beta^*)^k p^k_t||^2_{L^2(\Delta_k(1) \times \RR^k)}. $$
Assumptions (a) and (b) yields $v_t \gamma_t^2 \sim t^{-(1+1/\alpha)}$. By Lemma \ref{ConvergenceLemma}, the summand tends to zero and is dominated by a summable sequence, so (\ref{EqThm1}) is valid by the dominated convergence. We therefore finish the proof.
\endproof

\subsection{Proof of Theorem \ref{main2}} 
\proof 
The proof is almost the same as that of Theorem \ref{main}, so we only sketch it.

\begin{itemize}
\item[(i)] 
    Let $g \in L^2(\Delta_k(T) \times \RR^k)$, for $k \geq 1, T > 0$. We define
    $$\tilde{g}_t (\s,\x) :=  g (\s / t, \x / t^{1/\alpha}), $$
     which is a function on $\Delta_k(tT) \times \RR^k$. Following the lines of the proof of Proposition \ref{k-thItemConvergence}, we can show that
    as $t \rightarrow \infty,$
    \begin{equation*}
        T^{k(1+\frac{1}{\alpha})} \gamma_{tT}^k \bar{\omega}_{tT}^{(k)} (\tilde{g}_t) \claw  I_{k,T}(g).
    \end{equation*}
\item[(ii)] Let $(g^k)_{k\geq 0}$ be a family of functions defined on $\Delta_k(T)\times \RR^k$. If
 $\sum_{k = 0}^\infty || g^k ||_{\Delta_k(T)\times \RR^k}^2 < \infty$, then the sum $ \sum_{k=0}^{\infty} T^{k(1+\frac{1}{\alpha})}  \gamma_{tT}^k \bar{\omega}_{tT}^{(k)} (\tilde{g}_t^k)$ is well defined, and as
    $t \rightarrow \infty,$
$$ \sum_{k=0}^{\infty}  T^{k(1+\frac{1}{\alpha})}  \gamma_{tT}^k \bar{\omega}_{tT}^{(k)} (\tilde{g}_t^k) \claw \sum_{k=0}^{\infty} I_{k,T}(g^k).$$
The above is a generalization of Corollary \ref{CoroConvergence}.
\item[(iii)]

As $\epsilon \rightarrow 0$, 
 $$ \epsilon^{-k} \EE_{\P}\lcr \prod_{i = 1}^{k} 
\chi_{s_i,x_i}^{\epsilon}(X) | X_T = Y \rcr p(T, Y) \rightarrow p^k(\s, \x; T, Y), $$
for all $ \s \in \Delta_k(T) $ and all $\x \in \RR^k.$
Moreover, we can prove an analogue of Lemma \ref{lemmabound} by modifying $H_k$, which can be found easily following the lines of the proof of Lemma \ref{lemmabound}. The integrability of $H_k$  can be deduced from the proof of Proposition \ref{SolutionSHE}. 
\item[(iv)]
Define 
\begin{equation*}
    \psi_t^k(\s,\x; T,Y) = \gamma_{tT}^{-k}\lambda(\beta_{tT})^k p(T,Y) \EE_{\P}\lcr \prod_{i = 1}^{k} \chi_{s_i,x_i}^{r_{tT}/t^{1/\alpha}}(X) | X_T = Y \rcr {\bf 1}_{\Delta_k(T)}(\s).
\end{equation*}
For all $k \geq 0$, as $t \rightarrow \infty$,
    $$ ||\psi_t^k(\cdot,\cdot; T,Y)  - T^{k/\alpha} (\beta^*)^k p^k(\cdot,\cdot; T,Y)||_{L^2(\Delta_k(T) \times \RR^k)} \rightarrow 0.$$ 
Moreover, we can prove an analogue of Lemma \ref{ConvergenceLemma} (2), in which the upper bound is given by the integral of $H_k$ in (iii). The bound is also a summable sequence. 
\item[(v)] 
Using again Proposition \ref{chaosexpansion}, similar with Proposition \ref{WienerIto}, we have
\begin{align*}
    &  t^{1/\alpha}  W_{tT, t^{1/\alpha} Y}(\omega^{v_{tT}}, \beta_{tT}, r_{tT}) \\
    &=  t^{1/\alpha} \sum_{k=0}^{\infty} \frac{1}{k!}\bar{\omega}^{(k)}_{tT} \lpa \lambda(\beta_{tT})^k \EE_{\P}\lcr \prod_{i = 1}^{k} \chi_{s_i,x_i}^{r_{tT}}(X) \, | \,  X_{tT} =  t^{1/\alpha} Y  \rcr p(tT, t^{1/\alpha} Y) \rpa \\
   &= \sum_{k=0}^{\infty} \gamma_{tT}^k \bar{\omega}_{tT}^{(k)} (\tilde{\psi}_t^k(\cdot, \cdot; T,Y)),
\end{align*}
where 
\begin{align*}
   \tilde{\psi}_t^k(\s,\x; T,Y) &= \psi_t^k(\s /t,\x /t^{1/\alpha}; T, Y) \\
   &= \gamma_{tT}^{-k}\lambda(\beta_{tT})^k p(T, Y) \EE_{\P}\lcr \prod_{i = 1}^{k} \chi_{s_i/t,x_i/t^{1/\alpha}}^{r_{tT}/t^{1/\alpha}}(X) | X_{T} = Y \rcr {\bf 1}_{\Delta_k(tT)}(\s).
\end{align*}
\end{itemize}

By the above (ii), we have 
$$ \sum_{k=0}^{\infty} \gamma_{tT}^k \bar{\omega}_{tT}^{(k)}( T^{k/\alpha}  (\beta^*)^k \tilde{p}^k_t(\cdot, \cdot; T,Y))\claw \sum_{k=0}^{\infty}  I_{k,T}((\beta^*/T)^k p^k(\cdot, \cdot; T,Y))  = \cZ_{\alpha, \beta^*/T} (T, Y), $$
where 
$$\tilde{p}_t^k(\s,\x; T,Y) = p^k(\s /t,\x /t^{1/\alpha}; T, Y). $$

Then it remains to prove a convergence similar with (\ref{EqThm1}), namely

\begin{equation*}
    ||  \sum_{k=0}^{\infty} \gamma_{tT}^k \bar{\omega}_{tT}^{(k)} (\tilde{\psi}_t^k(\cdot, \cdot; T,Y)) - \sum_{k=0}^{\infty} \gamma_{tT}^k \bar{\omega}_{tT}^{(k)}( T^{k/\alpha}  (\beta^*)^k \tilde{p}^k_t(\cdot, \cdot; T,Y)) ||_2^2  \rightarrow 0,  \, t \rightarrow \infty. 
\end{equation*}
Using exactly the same argument, with the help of (i)-(v), we can finish the proof. 
\endproof

%\subsection{Proof of Proposition \ref{WienerIto} }

%\subsection{Proof of Proposition \ref{Zalphabeta}}

%\subsection{Proof of Proposition \ref{k-thItemConvergence}}  \hspace{1cm}   \\

\bigskip

\noindent
\textbf{Acknowledgements.}
I would like to thank Prof. Fuqing Gao for many fruitful discussions.

\bibliographystyle{plain} 
\bibliography{SDP}
\end{document}